\documentclass[10pt]{article}
\usepackage{amsmath,amssymb,amsfonts}
\usepackage{epsf,color}
\usepackage{fancyhdr}
\usepackage{blindtext}
\usepackage{url}
\newtheorem{theorem}{Theorem}

\newtheorem{lemma}{Lemma}

\date{}

\title{On the Deepest Cycle of a Random Mapping}
\bigskip

\author{  {\sc Ljuben Mutafchiev}\thanks{Corresponding author; E-mail: Ljuben@aubg.edu}\,\,\thanks{Also at: Institute of Mathematics
and Informatics, Bulgarian Academy of Sciences,
  Sofia 1113, Bulgaria} \\ \small American University in Bulgaria,
  2700 Blagoevgrad, Bulgaria
  \and
 {\sc Steven Finch}
 \\ \small MIT Sloan School of Management, Cambridge, MA 02142, USA}
\date{}
\begin{document}
\maketitle

\begin{abstract}
Let $\mathcal{T}_n$ be the set of all mappings
$T:\{1,2,\ldots,n\}\to\{1,2,\ldots,n\}$. The corresponding graph
of $T$ is a union of disjoint connected unicyclic components. We
assume that each $T\in\mathcal{T}_n$ is chosen uniformly at random
(i.e., with probability $n^{-n}$). The cycle of $T$ contained
within its largest component is called the {\it deepest} one. For
any $T\in\mathcal{T}_n$, let $\nu_n=\nu_n(T)$ denote the length of
this cycle. In this paper, we establish the convergence in
distribution of $\nu_n/\sqrt{n}$ and find the limits of its
expectation and variance as $n\to\infty$. For $n$ large enough, we
also show that nearly $55\%$ of all cyclic vertices of a random
mapping $T\in\mathcal{T}_n$ lie in its deepest cycle and that a
vertex from the longest cycle of $T$ does not belong to its
largest component with approximate probability $0.075$.
\end{abstract}

\vspace{.5cm}



\vspace{.2cm}

\section{Introduction and Statement of the Main Result}

We start with some notation that will be used freely in the text.

For a positive integer $n$, let $\mathcal{T}_n$ denote the set of
all mappings $T:[n]\to [n]$, where $[n]:=\{1,2,\ldots,n\}$. It is
clear that the cardinality $|\mathcal{T}_n|$ of $\mathcal{T}_n$ is
$n^n$. A mapping $T\in\mathcal{T}_n$ corresponds to a directed
graph $G_T$, called a functional digraph, with edges $(i,T(i)),
i\in [n]$, where every vertex $i\in [n]$ has out-degree $1$. $G_T$
is a union of disjoint connected components. A vertex $i$ is
called cyclic if, for the $m$-fold composition $T^{(m)}$ of $T$,
we have $T^{(m)}(i)=i$ for some $m\ge 1$. Since the vertices of
$G_T$ have out-degree $1$, each component contains a unique
directed cycle and directed trees attached to the cyclic vertices.
Let $\lambda_n=\lambda_n(T), T\in\mathcal{T}_n$, denote the number
of the cyclic vertices in $G_T$. We introduce the uniform
probability measure $\mathbb{P}$ on the set $\mathcal{T}_n$. That
is, we assign the probability $n^{-n}$ to each
$T\in\mathcal{T}_n$. In this way, $\lambda_n$, as well as any
other numerical characteristic of $G_T$, becomes a random variable
(or, a statistic in the sense of random generation of mappings
from $\mathcal{T}_n$). The size of the largest component of $G_T$
will be further denoted by $\mu_n=\mu_n(T)$. The cycle contained
within the largest component of $G_T$ is called the {\it deepest}
one. Let $\nu_n=\nu_n(T)$ denote its length. In \cite{F22}, Finch
suggests to study the asymptotic behavior of $\nu_n$ as
$n\to\infty$. The main goal of this work is to establish the
limiting distribution of $\nu_n/\sqrt{n}$ and find the asymptotics
of its mean and variance as $n\to\infty$.

There is a substantial probabilistic literature on random
mappings. Here we give only references to the popular monographs
\cite{S97,K86,ABT03}. For large $n$, some properties of the
functional digraphs $G_T, T\in\mathcal{T}_n$, are also used in the
analysis of algorithms. For example, the cyclic structure of
random mappings is closely related to algorithms for integer
factorization and, in particular, to the Pollard's
$\rho$-algorithm; see, e.g., \cite{P75,BP81,FO90,KLMSS16}. A
comprehensive exposition on the algorithms of factorization of
integers and other related topics may be found in \cite[Section
4.5.4]{K98}. Random mapping statistics are also relevant to some
algorithms for generic attacks on iterated hash constructions;
see, e.g., \cite{BGW18}.

Throughout the paper, the notation $\mathbb{E}$ and $\mathbb{V}ar$
stand for the expectation and variance with respect to the uniform
probability measure $\mathbb{P}$ on the set $\mathcal{T}_n$,
respectively. We also denote the convergence in distribution by
$\to_d$. To state our main results, we need to introduce the
distribution functions of some random variables that will be used
further.

Let $\chi^2(1)$ be a chi-squared distributed random variable with
one degree of freedom, that is, $\chi^2(1)=\xi^2$, where $\xi$ has
the standard normal distribution function
$$
\Phi(x)=\frac{1}{\sqrt{2\pi}}\int_{-\infty}^x e^{-u^2/2}du.
$$
It can be easily checked that the distribution function of
$\chi^2(1)$ is given by
\begin{equation}\label{chi}
G(x) =2\Phi(\sqrt{x})-1, \quad x>0,
\end{equation}
and $G(x)=0$ elsewhere. More details on the chi-squared
distribution may be found, e.g., in \cite[Chapter II, Section
3]{F66}. We now turn to the size $\mu_n$ of the largest component
in a random mapping $T\in\mathcal{T}_n$. The limiting distribution
function of $\mu_n/n$ was first determined by Kolchin \cite{K76}
(see also \cite[Section 1.13]{K86}). To present the result, we
shall use some notation and facts from \cite[Section 5.5]{ABT03}.
Consider first a random variable $\eta$ whose probability density
function $p(x), x>0$, is given by
\begin{equation}\label{p}
p(x) =\frac{e^{-\gamma/2}}{\sqrt{\pi x}}\left(1+\sum_{k=1}^\infty
\frac{(-1)^k}{2^k k!}\int\ldots\int_{I_k(x)} (1-\sum_{j=1}^k
y_j)^{-1/2} \frac{dy_1\ldots dy_k}{y_1\ldots y_k}\right),
\end{equation}
where
\begin{equation}\label{ik}
I_k(x):=\{(y_1,\ldots, y_k):y_1>x^{-1},\ldots,y_k>x^{-1},
\sum_{j=1}^k y_j<1\},
\end{equation}
and $\gamma\approx 0.5772$ denotes Euler's constant. The integral
limit theorem for the size of the largest component of a random
mapping states that
\begin{equation}\label{mu}
\frac{\mu_n}{n}\to_d\mu
\end{equation}
as $n\to\infty$, where the random variable $\mu$ has the
distribution function $F$ given by
\begin{equation}\label{df}
F(x) =e^{\gamma/2}\sqrt{\pi/x}p(1/x), \quad x>0
\end{equation}
and $p(x)$ is as defined in (\ref{p}); see \cite[Lemma
5.7]{ABT03}.

We are now ready to state our results on the deepest cycle of a
random mapping.

\begin{theorem}
(i) As $n\to\infty$, $n^{-1/2}\nu_n\to_d\sqrt{\chi^2(1)\mu}$,
where $\chi^2(1)$ and $\mu$ are independent random variables with
distribution functions $G(x)$ and $F(x)$ given by (\ref{chi}) and
(\ref{df}), respectively.

(ii) Let $E_1(s)=\int_s^\infty\frac{e^{-t}}{t}dt, s>0,$ be the
exponential integral function. Then we have
$$
\lim_{n\to\infty}\frac{1}{\sqrt{n}}\mathbb{E}(\nu_n)
=\frac{1}{\sqrt{2}}\int_0^\infty\frac{\exp{(-s-\frac{1}{2}E_1(s)})}
{\sqrt{s}} ds\approx 0.6884.
$$

(iii) We also have

$$
\lim_{n\to\infty}\frac{1}{n}\mathbb{V}ar(\nu_n)\approx 0.2839.
$$
\end{theorem}

The proof of this theorem is given in Section 3. Section 2
contains some preliminary results. In Section 4, we consider a
sampling experiment: we assume that a vertex of the random
functional digraph $G_T, T\in\mathcal{T}_n$, is chosen uniformly
at random from the set $[n]$. We give interpretations arising from
this random choice. We conclude Section 4 with an open problem.

\section{Preliminary Results}

A mapping is {\it indecomposable} (or {\it connected}) if it
possesses exactly one component. Let
$\mathcal{T}_n^\prime\subset\mathcal{T}_n$ be the subset of
indecomposable mappings of $[n]$ into itself. The cardinality
$|\mathcal{T}_n^\prime|$ of the set $\mathcal{T}_n^\prime$ was
determined by Katz \cite{K55}, who showed that
$$
|\mathcal{T}_n^\prime| =(n-1)!\sum_{k=0}^{n-1}\frac{n^k}{k!}.
$$
For the sake of convenience, we set
\begin{equation}\label{a}
A_n:=\frac{|\mathcal{T}_n^\prime|}{(n-1)!}
=\sum_{k=0}^{n-1}\frac{n^k}{k!}.
\end{equation}
Multiplying the right-hand side of (\ref{a}) by $e^{-n}$, we
observe that it represents a sum of probabilities of a
Poisson-distributed random variable with mean $n$. Applying the
normal approximation of the Poisson distribution in conjunction
with the Berry-Esseen bound on the rate of convergence in this
approximation (see, e.g., \cite[Chapter XVI, Section 5]{F66}), we
conclude that
\begin{equation}\label{indec}
e^{-n}A_n=\frac{1}{2}+O\left(\frac{1}{\sqrt{n}}\right) \quad
\text{as} \quad n\to\infty.
\end{equation}

Now, we introduce the uniform probability measure $\mathcal{P}$ on
the set $\mathcal{T}_n^\prime$. Let
$\nu_n^\prime=\nu_n^\prime(T),T\in\mathcal{T}_n^\prime$, denote
the count of the cyclic vertices in $T$. R$\acute{e}$nyi
\cite{R59} showed that, with respect to $\mathcal{P}$,
$\nu_n^\prime/\sqrt{n}$ converges in distribution to the random
variable $|\xi|$, where $\xi$ has a standard normal distribution.
In addition, he also established a local limit theorem for
$\nu_n^\prime/\sqrt{n}$ and showed that
$\mathcal{E}(\nu_n^\prime)\sim\sqrt{\frac{2n}{\pi}}$ as
$n\to\infty$, where $\mathcal{E}$ denotes the expectation with
respect to the probability measure $\mathcal{P}$. The next lemma
establishes a more precise estimate for this expectation.

\begin{lemma} As $n\to\infty$, we have
$$
\mathcal{E}(\nu_n^\prime) =\sqrt{\frac{2n}{\pi}}+O(1).
$$
\end{lemma}

\begin{proof} To find the probability mass function of
$\nu_n^\prime$, we notice first that each graph $G_T,
T\in\mathcal{T}_n^\prime$, is a cycle of trees rooted at the
cyclic vertices of $G_T$. The number of forests on $n$ vertices
containing $k$ rooted trees is ${n-1\choose{k-1}}n^{n-k}$, where
$1\le k\le n$ (see, e.g., \cite[Chapter 6, formula (0.5)]{S97}.
Multiplying this by $(k-1)!$ (the number of ways to construct a
cycle on $k$ vertices) and dividing the product by
$|\mathcal{T}_n^\prime|$, we see that
\begin{equation}\label{pmf}
\mathcal{P}(\nu_n^\prime=k) =\frac{n^{n-k}}{(n-k)!A_n}, \quad 1\le
k\le n,
\end{equation}
where $A_n$ is the quantity given by (\ref{a}). Since the
probability distributions on $\mathcal{T}_n$ and
$\mathcal{T}_n^\prime$ are both uniform, we conclude that
$$
\mathcal{E}(\nu_n^\prime)
=\frac{|\mathcal{T}_n|}{|\mathcal{T}_n^\prime|}.
$$
This simple intuitive formula is also confirmed by a
straightforward computation which uses (\ref{pmf}). It shows that
\begin{equation}\label{enupr}
\mathcal{E}(\nu_n^\prime) =\sum_{k=1}^n \frac{kn^{n-k}}{(n-k)!A_n}
=\frac{n^n}{(n-1)!A_n}.
\end{equation}
Applying Stirling's formula and (\ref{indec}) to the right-hand
side of (\ref{enupr}), we obtain
$$
\mathcal{E}(\nu_n^\prime)
=\sqrt{\frac{2n}{\pi}}\left(1+O\left(\frac{1}{\sqrt{n}}\right)\right)
=\sqrt{\frac{2n}{\pi}}+O(1), \quad n\to\infty,
$$
as required. \hfill $\Box$
\end{proof}

We show next that an elementary argument leads to a surprisingly
simple exact expression for the second moment of $\nu_n^\prime$.

\begin{lemma}
For any fixed $m\ge 2$, we have $\mathcal{E}(\nu_m^{\prime 2})=m$.
\end{lemma}

\begin{proof}
Direct calculation shows that
\begin{eqnarray}
& & \mathcal{E}(\nu_m^{\prime 2})
=\frac{1}{A_m}\sum_{k=1}^m\frac{k^2 m^{m-k}}{(m-k)!}
=\frac{1}{A_m}\sum_{l=0}^{m-1}\frac{(m^2-2ml+l^2)m^l}{l!}
\nonumber \\
& & =\frac{1}{A_m}\left(\sum_{l=0}^{m-1}\frac{m^{l+2}}{l!}
-2\sum_{l=1}^{m-1}\frac{m^{l+1}}{(l-1)!}
+\sum_{l=2}^{m-1}\frac{m^l}{(l-2)!}
+\sum_{l=1}^{m-1}\frac{m^l}{(l-1)!}\right) \nonumber \\
& & =\frac{1}{A_m}\left(\frac{m^{m+1}}{(m-1)!}
+\sum_{l=2}^m\frac{m^l}{(l-2)!}
-2\sum_{l=2}^m\frac{m^l}{(l-2)!}\right)
\nonumber \\
& & +\frac{1}{A_m}\left(\sum_{l=2}^m\frac{m^l}{(l-2)!}
-\frac{m^m}{(m-2)!} +\sum_{l=1}^{m-1}\frac{m^l}{(l-1)!}\right)
\nonumber \\
& & =\frac{1}{A_m}\left(\frac{m^m}{(m-1)!}
+\sum_{l=1}^{m-1}\frac{m^l}{(l-1)!}\right)
=\frac{1}{A_m}\sum_{l=1}^{m}\frac{m^l}{(l-1)!} \nonumber \\
& & =\frac{m}{A_m}\sum_{l=1}^{m}\frac{m^{l-1}}{(l-1)!}
=\frac{m}{A_m}\sum_{k=0}^{m-1}\frac{m^k}{k!} =m, \nonumber
\end{eqnarray}
where in the last equality we have used (\ref{a}). This completes
the proof of the lemma. \hfill  $\Box$
\end{proof}

Now, we return to the set $\mathcal{T}_n$ of unrestricted mappings
of $[n]$ into itself. In the computation of the first two moments
of $\nu_n$, we shall use conditional expectations. It is clear
that, for $m\le n$ and $j=1,2$, we have
$\mathbb{E}(\nu_n^j|\mu_n=m)=\mathcal{E}(\nu_m^{\prime j})$, where
$\mu_n$ stands for the largest component size of a mapping from
$\mathcal{T}_n$. Decomposing  $\mathbb{E}(\nu_n^j)$ into a
weighted sum of conditional expectations, we have
\begin{equation}\label{con}
\mathbb{E}(\nu_n^j) =\sum_{m=1}^n\mathbb{E}(\nu_n^j|\mu_n=m)
\mathbb{P}(\mu_n=m)=\sum_{m=1}^n\mathcal{E}(\nu_m^{\prime j})
\mathbb{P}(\mu_n=m).
\end{equation}
Hence, setting $j=2$ and applying Lemma 2, we obtain, for all
$n\ge 1$, the following curious identity:
\begin{equation}\label{id}
\mathbb{E}(\nu_n^2) =\mathbb{E}(\mu_n).
\end{equation}

Now, we proceed to the preliminaries concerning the largest
component of a random mapping. Arratia et al. \cite{ABT03}
developed a unifying approach to the study of the component
spectrum of a large parametric class of decomposable combinatorial
structures, called logarithmic structures. These structures
satisfy a condition, called there logarithmic. It introduces a
dependency on a parameter $\theta>0$. More precisely, this general
approach employs a random variable $\eta_\theta$ (see \cite[Lemma
4.7]{ABT03}). Its probability density function $p_\theta(x), x>0$,
is given by
\begin{equation}\label{pth}
p_\theta(x) =\frac{e^{-\gamma\theta}x^{\theta-1}}{\Gamma(\theta)}
\left(1+\sum_{k=1}^\infty \frac{(-1)^k}{2^k
k!}\int\ldots\int_{I_k(x)} (1-\sum_{j=1}^k y_j)^{-1/2}
\frac{dy_1\ldots dy_k}{y_1\ldots y_k}\right),
\end{equation}
where $\Gamma(.)$ denotes Euler's gamma function and $I_k(x)$ are
defined by (\ref{ik}). Setting $\theta=1/2$, we obviously obtain
the expression for $p(x)$ given by (\ref{p}). It is shown
\cite[Section 6.1]{ABT03} that random mappings satisfy the
logarithmic condition with this value of $\theta$. For the sake of
simplicity, in formulas (\ref{p}) and (\ref{df}) given in the
Introduction as well as in the material that we shall present
further, we omit the index $\theta$ and restrict ourselves to the
case of random mappings, where $\theta=1/2$. For more details, we
refer the reader to \cite[Sections 4.2 and 5.5]{ABT03}.

The next lemma provides a formula for the Laplace transform of the
function $p(x)$; see (\ref{p}). It is a particular case of Theorem
4.6 from \cite{ABT03}.

\begin{lemma}
We have
\begin{equation}\label{lapl}
\varphi(s):=\int_0^\infty e^{-sx}p(x)dx
=\exp{\left(-\frac{1}{2}\int_0^1\frac{1-e^{-sy}}{y}dy\right)}
=\frac{e^{-\gamma/2}}{\sqrt{s}}e^{-\frac{1}{2}E_1(s)},
\end{equation}
where $p(x)$ is given by (\ref{p}) and $E_1(s)$ denotes the
exponential integral function introduced in Theorem 1(ii).
\end{lemma}

We only notice here that the last equality in (\ref{lapl}) follows
from the classical identity
$$
\int_0^s\frac{1-e^{-y}}{y}dy =E_1(s)+\log{s}+\gamma, \quad s>0;
$$
see, e.g., \cite[Section 5.1]{AS65}.

Watterson \cite{W76} observed that $p(x)$ satisfies the delay
differential equation
\begin{equation}\label{dde}
xp^\prime(x)+\frac{1}{2}p(x)+\frac{1}{2}p(x-1)=0\quad\text\quad
\text{for}\quad x>1,\quad p(x)=\frac{e^{-\gamma/2}}{\sqrt{\pi
x}}\quad\text{for}\quad 0<x\le 1.
\end{equation}

{\it Remark.} For decomposable combinatorial structures with
parameter $\theta$, the delay differential equation (\ref{dde})
becomes
$$
xp_\theta^\prime(x)+(1-\theta)p_\theta(x) +\theta p_\theta(x-1)=0,
$$
where $p_\theta(x)$ is the probability density function given by
(\ref{pth}). If $\theta=1$, then $p_1(1/x)$ is a distribution
function which describes the asymptotic behavior of the largest
prime factor of a random integer and the size of the longest cycle
of a random permutation of $n$ letters. For more details, we refer
the reader to \cite{V86,VS77}.

From (\ref{df}) and (\ref{dde}) one can easily deduce the limiting
probability density function $f(x)$ of $\mu_n/n$. We have
\begin{equation}\label{f}
f(x)=\frac{1}{2}e^{\gamma/2}\sqrt{\pi}x^{-3/2}
p\left(\frac{1}{x}-1\right), \quad 0<x\le 1.
\end{equation}
In the Introduction, we already stated the integral limit theorem
for $\mu_n/n$; see (\ref{mu}) and (\ref{df}). Arratia et al.
\cite[Lemma 5.9]{ABT03} derived also a local limit theorem for
$\mu_n$, which we shall essentially use in the proof of Theorem 1.
It is stated as follows.

\begin{lemma}
Suppose that $m\le n$ satisfies
$\frac{m}{n}\to x\in (0,1)$ as $n\to\infty$. Then
$$
\mathbb{P}(\mu_n=m)=\frac{1}{n}f(x)(1+o(1)),
$$
where $f(x)$ is given by (\ref{f}).
\end{lemma}

\section{Proof of Theorem 1}

{\it Proof of part (i).} The essential part of the argument in
this proof is based on the relationship between the probabilities
$\mathbb{P}$ defined on the set $\mathcal{T}_n$, conditioned upon
events $\{\mu_n=m\}$, and the unconditional probabilities
$\mathcal{P}$ defined on the subset
$\mathcal{T}_n^\prime\subset\mathcal{T}_n$. In particular, for any
$1\le m\le n$ and $1\le N\le m$, we have $\mathbb{P}(\nu_n\le
N|\mu_n=m)=\mathcal{P}(\nu_m^\prime\le N)$. The total probability
formula implies that
\begin{eqnarray}\label{nudf}
& & \mathbb{P}(\nu_n\le N) =\sum_{m=1}^n\mathbb{P}(\nu_n\le
N|\mu_n=m)\mathbb{P}(\mu_n=m) \nonumber \\
& & =\sum_{m=1}^n\mathcal{P}(\nu_m^\prime\le N)\mathbb{P}(\mu_n=m)
=\Sigma_1+\Sigma_2,
\end{eqnarray}
where in the last equality we have decomposed the underlying
probability into two sums in the following way. For a function
$\omega(n)$ satisfying $\omega(n)\to\infty$ and $\omega(n)=o(n)$
as $n\to\infty$, we set
\begin{equation}\label{nuone}
\Sigma_1=\sum_{1\le m\le\omega(n)}\mathcal{P}(\nu_m^\prime\le
N)\mathbb{P}(\mu_n=m),
\end{equation}

\begin{equation}\label{nutwo}
\Sigma_2=\sum_{\omega(n)< m\le n}\mathcal{P}(\nu_m^\prime\le
N)\mathbb{P}(\mu_n=m).
\end{equation}
For $\Sigma_1$, we easily obtain the bound
\begin{equation}\label{sigone}
\Sigma_1\le\sum_{1\le m\le\omega(n)}\mathbb{P}(\mu_n=m)
=\mathbb{P}(\mu_n\le\omega(n))
=\mathbb{P}\left(\frac{\mu_n}{n}\le\frac{\omega(n)}{n}\right).
\end{equation}
In view of (\ref{mu}) and (\ref{df}), we approximate the last
probability in (\ref{sigone}) by $F(\omega(n)/n)$. To complete the
estimate, we use that $\lim_{x\to\infty}xp(x)=0$; see \cite[p.
85]{ABT03}. Setting in (\ref{df}) $x=n/\omega(n)$, we obtain
$$
F(\omega(n)/n) =\frac{e^{\gamma/2}\sqrt{\pi}}{\sqrt{n/\omega(n)}}
(n/\omega(n))p(n/\omega(n))
=o\left(\sqrt{\frac{\omega(n)}{n}}\right),
$$
since $n/\omega(n)\to\infty$ as $n\to\infty$. Combining this with
(\ref{sigone}), we conclude that
\begin{equation}\label{oneest}
\Sigma_1 =o\left(\sqrt{\frac{\omega(n)}{n}}\right)=o(1), \quad
n\to\infty.
\end{equation}
In the second sum $\Sigma_2$, we replace first the probabilities
$\mathcal{P}(\nu_m^\prime=k)$ by their probability mass
expressions given by (\ref{pmf}). Thus $\Sigma_2$ is converted to
\begin{eqnarray}\label{sigtwo}
& & \Sigma_2= \sum_{\omega(n)<m\le n}\mathbb{P}(\mu_n=m)\sum_{1\le
k\le
N}\mathcal{P}(\nu_m^\prime=k) \nonumber \\
& & =\sum_{\omega(n)<m\le n}\mathbb{P}(\mu_n=m)\sum_{1\le k\le N}
\frac{m^{m-k}}{(m-k)!A_m} \nonumber \\
& & =\sum_{\omega(n)<m\le n}\mathbb{P}(\mu_n=m)\sum_{m-N\le l\le
m-1} \frac{m^l}{l!A_m}.
\end{eqnarray}
Further on, for any $a>0$, by $Po(a)$, we denote a
Poisson-distributed random variable with mean $a$. Moreover, let
$\mathbf{P}$ be the Lebesgue measure on the Borel subsets of
$[0,\infty)$. To estimate the inner sum in the last line of
(\ref{sigtwo}), we use (\ref{indec}) and apply again the normal
approximation of the Poisson distribution in conjunction with the
Berry-Esseen bound on the rate of convergence in this
approximation (see, e.g., \cite[Chapter XVI, Section 5]{F66}).
Since $m>\omega(n)\to\infty$, we have
\begin{eqnarray}\label{in}
& & \sum_{m-N\le l\le m-1} \frac{m^l e^{-m}}{l!A_m e^{-m}}
=\frac{1}{1/2+O(1/\sqrt{m})}\mathbf{P}(m-N\le Po(m)<m) \nonumber
\\
& & =2(1+O(1/\sqrt{m})) \mathbf{P}\left(-\frac{N}{\sqrt{m}}\le
\frac{Po(m)-m}{\sqrt{m}}<0\right) \nonumber \\
& & =(1+o(1))(2\Phi(N/\sqrt{m})-1),
\end{eqnarray}
where $\Phi(x)$ is the standard normal distribution function.
Going back to (\ref{sigtwo}), we apply the result of Lemma 4 and
replace the inner sum of its last line by the right-hand side
expression of (\ref{in}). For any $y>0$, we also set
\begin{equation}\label{y}
N=\sqrt{yn}.
\end{equation}
Thus, recalling (\ref{chi}), we find that
\begin{eqnarray}\label{twoest}
& & \Sigma_2 =(1+o(1))\sum_{\omega(n)/n<m/n\le 1}\frac{1}{n}
f(m/n)(2\Phi(\sqrt{y}/\sqrt{m/n})-1) \nonumber \\
& & =\int_0^1 f(x)G(y/x)dx +o(1),
\end{eqnarray}
since the sum in the middle of (\ref{twoest}) is a Riemann sum
with step size $1/n$ of the integral in its right-hand side. Now,
combining (\ref{nudf}) - (\ref{nutwo}), (\ref{oneest}), (\ref{y})
and (\ref{twoest}), we observe that
\begin{equation}\label{dfend}
\mathbb{P}\left(\frac{\nu_n}{\sqrt{n}}\le\sqrt{y}\right) =\int_0^1
f(x)G(y/x)dx +o(1).
\end{equation}
It is easily seen that the left-hand side probability in
(\ref{dfend}) represents the distribution function of the ratio
$\nu_n^2/n$. Furthermore, (\ref{chi}) and (\ref{f}) imply that the
integral in the right-hand side equals the distribution function
of $\chi^2(1)\mu$, where the random variables in this product are
independent. In other words, we find from (\ref{dfend}) that
$\nu_n^2/n\to_d\chi^2(1)\mu$, which obviously gives
$\nu_n/\sqrt{n}\to_d\sqrt{\chi^2(1)\mu}$ and completes the proof
of part (i). \hfill $\Box$

{\it Proof of part (ii).} In (\ref{con}) we first set $j=1$ and
then divide its first and third part by $\sqrt{n}$. Breaking up
the range of summation in it, we write
\begin{equation}\label{ss}
\frac{1}{\sqrt{n}}\mathbb{E}(\nu_n)=\frac{1}{\sqrt{n}}(S_1+S_2),
\end{equation}
where
\begin{equation}\label{sone}
S_1 =\sum_{1\le m\le\frac{1}{4}\log{n}} \mathcal{E}(\nu_m^\prime)
\mathbb{P}(\mu_n=m)
\end{equation}
and
\begin{equation}\label{stwo}
S_2=\sum_{\frac{1}{4}\log{n}<m\le n} \mathcal{E}(\nu_m^\prime)
\mathbb{P}(\mu_n=m).
\end{equation}
As in the proof of part (i), $Po(m)$ denotes a Poisson-distributed
random variable with mean $m$. Recall also that $\mathbf{P}$
denotes the Lebesgue measure on the Borel subsets of $[0,\infty)$.
Then, for any $m$ satisfying $1\le m\le n$, in view of (\ref{a})
and (\ref{enupr}) we have
\begin{equation}\label{num}
\mathcal{E}(\nu_m^\prime)=\frac{m^m
e^{-m}}{(m-1)!\mathbf{P}(Po(m)\le m-1)}.
\end{equation}
To obtain an upper bound for $\mathcal{E}(\nu_m^\prime)$, we
proceed as follows. We multiply the numerator and the denominator
in the right-hand side of (\ref{num}) by $\sqrt{2\pi
m}e^{1/(12m+1)}$ and apply the following well-known inequality for
$m!$:
$$
m!>m^m e^{-m}\sqrt{2\pi m}e^{1/(12m+1)};
$$
see, for example, \cite[Chapter II, Section 9]{F57}. Furthermore,
note that, for $m\ge 1$, the Poisson probability in the
denominator of (\ref{num}) is uniformly bounded from below by
$\mathbf{P}(Po(m)=0)=e^{-m}$. Hence, for all $m, 1\le m\le n$, we
have
$$
\mathcal{E}(\nu_m^\prime)<\frac{m!e^m}{(m-1)!\sqrt{2\pi
m}e^{1/(12m+1)}}< \sqrt{\frac{m}{2\pi}}e^m.
$$
It follows from this, for large enough $n$, that
\begin{eqnarray}\label{neg}
& & \frac{1}{\sqrt{n}}S_1\le\frac{1}{\sqrt{2\pi}}
\sum_{m\le\frac{1}{4}\log{n}}\sqrt{\frac{m}{n}}
e^{\frac{1}{4}\log{n}}\mathbb{P}(\mu_n=m) \nonumber
\\
& & \le\frac{1}{2\sqrt{2\pi
n}}n^{1/4}\sqrt{\log{n}}\sum_{m\le\frac{1}{4}\log{n}} 1 =
O(n^{-1/4}\log^{3/2}{n})=o(1).
\end{eqnarray}
The estimate of $\frac{1}{\sqrt{n}}S_2$ follows from Lemmas 1 and
4. For large enough $n$, we have
\begin{eqnarray}\label{stwoest}
& & \frac{1}{\sqrt{n}} S_2 =\sqrt{\frac{2}{\pi}}
\sum_{\frac{1}{4}\log{n}<m\le n} \sqrt{\frac{m}{n}} \frac{1}{n}
f\left(\frac{m}{n}\right) (1+o(1)) \nonumber \\
& & +O\left(\frac{1}{\sqrt{n}} \sum_{\frac{1}{4}\log{n}<m\le n}
\frac{1}{n} f\left(\frac{m}{n}\right)\right).
\end{eqnarray}
The sums in the right-hand side of (\ref{stwoest}) are Riemann
sums of the integrals
\begin{equation}\label{i}
I:=\sqrt{\frac{2}{\pi}}\int_0^1\sqrt{x}f(x)dx
\end{equation}
and
\begin{equation}\label{one}
\int_0^1f(x)dx=1
\end{equation}
with step size $1/n$. Hence, combining (\ref{ss}) - (\ref{stwo})
and (\ref{neg}) - (\ref{one}), we obtain
\begin{equation}\label{enu}
\frac{1}{\sqrt{n}}\mathbb{E}(\nu_n)=I+o(1), \quad n\to\infty.
\end{equation}
To complete the proof, it remains to evaluate the integral $I$. We
first replace $f(x)$ by its expression (\ref{f}). Then, we set in
(\ref{i}) $y=\frac{1}{x}-1$. Recall that $p(x)$, defined in the
Introduction by (\ref{p}), is the probability density function of
a random variable $\eta>0$. Hence we can rewrite (\ref{i}) as
follows:
\begin{equation}\label{me}
I=\frac{e^{\gamma/2}}{\sqrt{2}}\int_0^\infty\frac{p(y)}{1+y}dy
=\frac{e^{\gamma/2}}{\sqrt{2}}\mathbf{E}((1+\eta)^{-1}).
\end {equation}
The Laplace transform $\varphi(s)$ of $\eta$ is given in Lemma 3
by both right side expressions of (\ref{lapl}). Furthermore, an
obvious computation shows that
\begin{equation}\label{mel}
\int_0^\infty e^{-s}\varphi(s)ds=\mathbf{E}((1+\eta)^{-1}).
\end{equation}
Combining the last expression for $\varphi(s)$ in (\ref{lapl})
with (\ref{me}) and (\ref{mel}), we obtain the required
representation of $I$. Using Mathematica, version 12.0, we obtain
\begin{equation}\label{fin}
I=0.6884050874956\ldots.
\end{equation}
Combining (\ref{enu}) and (\ref{fin}) completes the proof of part
(ii). \hfill     $\Box$

{\it Proof of part (iii).} From (\ref{id}) we see that
$$
\lim_{n\to\infty}\frac{1}{n}\mathbb{E}(\nu_n^2)
=\lim_{n\to\infty}\frac{1}{n}\mathbb{E}(\mu_n)=0.7578230112\ldots,
$$
where the last numerical value was found by Gourdon \cite[p.
152]{G96} (see also \cite[Table 5.1]{ABT03}). Since
$\frac{1}{n}\mathbb{V}ar(\nu_n)
=\frac{1}{n}(\mathbb{E}(\nu_n^2)-(\mathbb{E}(\nu_n))^2)$, the
numerical result of part (iii) follows again from (\ref{enu}) and
(\ref{fin}). Thus the proof is complete. \hfill $\Box$

\section{Concluding Remarks}

As a subject of further research, we propose possible extensions
of both results of Theorem 1. First, the limit approximation of
$\nu_n/\sqrt{n}$ obtained in part (i) suggests in a natural way
the study of rate estimates. In our setting, this requires a
sharper bound on the error term of Lemma 4 (see also the proof of
Lemma 5.9 given in \cite[p. 112]{ABT03}). Some facts from the
theory of asymptotic decompositions related to the central limit
theorem \cite[Chapter XVI]{F66} could be also used. The result of
part (ii) may be strengthened by establishing a uniform asymptotic
expansion for $\mathbb{E}(\nu_n)$. Such an expansion allows one to
compute this expectation with great accuracy even if $n$ is not
too large. The unifying approach developed by Gourdon \cite{G96}
may be used here. The proof should also contain some necessary
details on the singularity analysis of the underlying generating
function. Both aforementioned extensions of Theorem 1 offer
serious technical difficulties. These problems may be a subject of
a subsequent study. In this paper, we prefer to keep the text as
simple as possible in order to make it accessible to a wider
audience.

We shall present now some interpretations which involve a sampling
experiment. Suppose that a vertex $i\in[n]$ of the graph $G_T,
T\in\mathcal{T}_n$, is chosen uniformly at random. The probability
that $i$ possesses a certain property (e.g., $i$ is a cyclic
vertex, $i$ belongs to the largest component of $G_T$, etc.) can
be computed directly, using the total probability formula. For
example, the probability that a randomly chosen vertex is cyclic
equals $\sum_{k=1}^n\frac{k}{n}\mathbb{P}(\lambda_n=k)
=\frac{1}{n}\mathbb{E}(\lambda_n)$ (recall that $\lambda_n$ is the
total number of cyclic vertices in $G_T$). In a similar way, one
can interpret the ratio $\mathbb{E}(\nu_n)/\mathbb{E}(\lambda_n)$
as the limiting conditional probability that a randomly chosen
cyclic vertex belongs to the largest component (deepest cycle). It
is well-known that $\mathbb{E}(\lambda_n)\sim\sqrt{\pi n/2}$ as
$n\to\infty$; for example, see \cite[Section 6.3]{S97}. Combining
this asymptotic equivalence with the numerical result of Theorem
1(i), we obtain the approximate value of this probability, namely,
\begin{equation}\label{cond}
\lim_{n\to\infty}\frac{\mathbb{E}(\nu_n)}{\mathbb{E}(\lambda_n)}
=\lim_{n\to\infty}\sqrt{\frac{2}{\pi n}} \mathbb{E}(\nu_n)\approx
0.5493.
\end{equation}

Now, consider the length $\kappa_n$ of the longest cycle of a
random mapping from $\mathcal{T}_n$. Purdom and Williams
\cite{PW68} obtained a general asymptotic formula for the $m$th
moment of $\kappa_n$, $m=1,2,\ldots$, and computed the first five
terms of the asymptotic expansion of $\mathbb{E}(\kappa_n)$. For
$1\le n\le 50$, they showed that the sum
$$
0.7824816 n^{1/2}+0.104055+0.0652068n^{-1/2}-0.1052117n^{-1}
+0.0416667n^{-3/2}
$$
is a good approximation for $\mathbb{E}(\kappa_n)$ and that
$\lim_{n\to\infty}\frac{1}{\sqrt{n}}\mathbb{E}(\kappa_n) \approx
0.7825$. Hence, the limiting conditional probability that a
randomly chosen cyclic vertex belongs to the longest cycle is
\begin{equation}\label{conl}
\lim_{n\to\infty}\frac{\mathbb{E}(\kappa_n)}{\mathbb{E}(\lambda_n)}
\approx 0.6243.
\end{equation}
The difference between (\ref{conl}) and (\ref{cond}) is
approximately equal to $0.075$. It can be interpreted as the
approximate limiting probability that the longest cycle and the
largest component of $G_T$ are disjoint. Finch \cite{F22} called the
component containing the longest cycle of a random mapping {\it
richest} component. In this terminology, the difference $0.075$
equals the approximate limiting probability that the richest
component is not the largest one. The problem concerning the average
size of the richest component remains unsolved.

Apropos our last remark, we propose another open problem related
to the size $\tau_n$ of the largest tree in a random mapping from
$\mathcal{T}_n$. Since $\tau_n$ does not exceed the size of the
component to which the largest tree belongs and $\mu_n$ is the
maximum component size of $T\in\mathcal{T}_n$, for all $n\ge 1$,
we have $\tau_n\le\mu_n$. The limiting distribution function of
$\tau_n/n$ as $n\to\infty$ was first determined by Stepanov
\cite{S69}. There is another probabilistic proof of this result
due to Pavlov \cite{P77} (see also \cite[Section 3.3]{K86}). The
following natural question arises: what can be said about the
probability that the largest tree is a subgraph of the largest
component of a random mapping? It seems the limit theorems from
\cite{S69,P77} would be helpful to obtain an asymptotic estimate
for this probability.

\section*{Acknowledgements}

We are grateful to the anonymous referees for their valuable
comments and suggestions on the earlier draft of this paper. The
first author is also grateful to Emil Kamenov and Mladen Savov for
the helpful discussions. He was partially supported by Project
KP-06-N32/8 with the Bulgarian Ministry of Education and Science.

\end{document}